\documentclass[english,11pt,leqno,twoside]{amsart}
\usepackage[T1]{fontenc}
\usepackage[latin9]{inputenc} 
\usepackage{amsmath}
\usepackage{amsthm}
\usepackage{amssymb}
\usepackage{amsfonts}

\usepackage{soul,color}
\usepackage{t1enc}
\usepackage{a4,indentfirst,latexsym,graphics,graphicx}
\usepackage{mathrsfs}
\usepackage{enumitem}
\usepackage[colorlinks=true,urlcolor=blue,
citecolor=red,linkcolor=blue,linktocpage,pdfpagelabels,
bookmarksnumbered,bookmarksopen]{hyperref}
\usepackage[english]{babel}
\usepackage[left=2.1cm,right=2.1cm,top=2.72cm,bottom=2.72cm]{geometry}
\usepackage[metapost]{mfpic}
\usepackage[hyperpageref]{backref}
\usepackage[colorinlistoftodos]{todonotes}
\usepackage[normalem]{ulem}
\usepackage{braket}

\usepackage{bbm}
\usepackage[noadjust]{cite}
\usepackage{mathtools} 
\usepackage{url}

\usepackage{lmodern}

\usepackage[useregional]{datetime2}

\usepackage{fixmath}
\usepackage{xspace}
\def\bb#1\eb{\textcolor{blue}
	{#1}} %
\def\br#1\er{\textcolor{red}
	{#1}} %
\def\bv#1\ev{\textcolor{green}
	{#1}} %
\def\bc#1\ec{\textcolor{cyan}
	{#1}} %

\def\bal#1\eal{\begin{align}#1\end{align}}                      %
\def\baln#1\ealn{\begin{align*}#1\end{align*}}          
\def\bml#1\eml{\begin{multline}#1\end{multline}}        %
\def\bmln#1\emln{\begin{multline*}#1\end{multline*}}  %
\def\bga#1\ega{\begin{gather}#1\end{gather}}
\def\bgan#1\egan{\begin{gather*}#1\end{gather*}}

\makeatletter

\@addtoreset{equation}{section} \makeatother
\newtheorem{theorem}{Theorem}[section]
\newtheorem{lemma}[theorem]{Lemma}

\newtheorem{remark}[theorem]{Remark}

\newcommand{\<}{\langle}
\renewcommand{\>}{\rangle}

\newcommand{\RR}{\mathbb R}
\newcommand{\CC}{\mathbb C}
\newcommand{\NN}{\mathbb N}

\def\R{{\mathbb R}}
\def\N{{\mathbb N}}

\renewcommand{\leq}{\leqslant}
\renewcommand{\ge}{\geqslant}
\renewcommand{\geq}{\geqslant}

\title[Magnetic fractional problems]{Asymptotically linear magnetic\\ fractional problems}

\author[R. Bartolo]{Rossella Bartolo}
\address{Rossella Bartolo\hfill\break\indent
	Dipartimento di Meccanica, Matematica e Management\hfill\break\indent
	Politecnico di Bari\hfill\break\indent
	Via E. Orabona 4, 70125 Bari\hfill\break\indent
	Italy}
\email{rossella.bartolo@poliba.it}

\author[P. d'Avenia]{Pietro d'Avenia}
\address{Pietro d'Avenia\hfill\break\indent
	Dipartimento di Meccanica, Matematica e Management\hfill\break\indent
	Politecnico di Bari\hfill\break\indent
	Via E. Orabona 4, 70125 Bari\hfill\break\indent
	Italy}
\email{pietro.davenia@poliba.it}

\author[G. Molica Bisci]{Giovanni Molica Bisci}
\address{Giovanni Molica Bisci\hfill\break\indent
	Dipartimento di Scienze Pure e Applicate (DiSPeA)\hfill\break\indent
	Universit\`a degli Studi di Urbino Carlo Bo\hfill\break\indent
	Piazza della Repubblica 13, 61029 Urbino\hfill\break\indent
	Italy}
\email{giovanni.molicabisci@uniurb.it}



\date{}

\begin{document}
	
	\begin{abstract}
		The aim of this paper is investigating the existence and multiplicity of weak solutions to non--local equations involving the {\em magnetic fractional Laplacian}, when the nonlinearity is subcritical and asymptotically linear at infinity. We prove existence and multiplicity results by using variational tools, extending to the magnetic local and non--local setting some known results for the classical and the fractional Laplace operators.\end{abstract}
	
	\maketitle

	\section{Introduction}
	Existence and multiplicity results for solutions of elliptic problems involving non--local operators have been
	faced by a large number of authors by using variational and topological methods also in view of applications; see, among others, the monograph \cite{MoRaSe} and the references therein.
	
	From a probabilistic point of view non--local operators can be seen as the infinitesimal generators of L\'{e}vy
	stable diffusion processes. Moreover, fractional operators allow us to model unusual diffusion processes in turbulent fluid motions and material transports in fractured media.\\
	In particular, the fractional Laplacian appears in generalizations of quantum mechanics and in the description of the motion of a chain or an array of particles that are connected by elastic springs (see \cite{Appl,Las,Poz}).
	
	Motivated by this wide interest in the current literature and by the meaning that the non--local operators can have in the applications, we are interested here in a nonlinear fractional problem involving an asymptotically linear term at infinity and the {\em fractional magnetic Laplacian}. Indeed, the main results (see Theorems \ref{main} and \ref{main1} below) give under quite general assumptions a non--local magnetic version of some previous results, already present in the current literature, that are valid for different classes of differential problems.
	
	More precisely, given $A\in C(\R^N,\R^N)$ we look for  solutions of the problem
	\begin{equation}
		\label{P}\tag{$P$}
		\begin{cases}
			(-\Delta)^s_Au=  g(x, u) & \mbox{in }  \Omega\\
			u=0 & \mbox{in }  \mathbb{R}^N\setminus \Omega
		\end{cases}
	\end{equation}
	where $\Omega$ is an open  bounded subset of $\mathbb{R}^N$ with Lipschitz boundary $\partial \Omega$, $N>2s$, $s\in ]0,1[$ and $(-\Delta)^s_A$ is the {\em fractional magnetic Laplacian} defined in \cite{DaSq}, generalizing an operator introduced in \cite{Ic} (see also \cite{IT}), as follows
	\begin{equation}
		\label{operator}
		(-\Delta)^s_Au(x):=c_{N,s}\lim_{\varepsilon\rightarrow 0^+}\int_{\R^N\setminus B_\varepsilon(x)}\frac{u(x)-e^{i(x-y)\cdot A(\frac{x+y}{2})}u(y)}{|x-y|^{N+2s}}dy \qquad x\in\RR^N
	\end{equation}
	with $u\in C_0^\infty(\mathbb{R}^N,\mathbb{C})$, $B_\varepsilon(x)$  ball of center $x$ and radius $\varepsilon$ and
	$$
	c_{N,s}= s2^{2s}\frac{{\rm \Gamma}\left(\frac{N+2s}{2}\right)}{\pi^\frac N2{\rm \Gamma}(1-s)}.
	$$
	If $N=3$, $B:=\nabla \times A$ physically represents an external magnetic field acting on a charged particle.
	
	When $A\equiv 0$ and $u\in C_0^\infty(\R^N,\RR)$, $(-\Delta)^s_A$ agrees with the standard fractional Laplacian $(-\Delta)^s$ defined as principal value integral
	$$
	(-\Delta)^su(x)=c_{N,s}
	\lim_{\varepsilon\rightarrow 0^+}\int_{\R^N\setminus B_\varepsilon(x)}\frac{u(x)-u(y)}{|x-y|^{N+2s}}dy\qquad x\in\RR^N,
	$$
	see e.g. \cite{dpv, MoRaSe}.
	
	Moreover, as observed in \cite{DaSq,Ic}, the operator defined in \eqref{operator}, can be seen as the fractional counterpart of the well known {\em  magnetic Laplacian} $(\nabla - iA)^2$, 
	see e.g. \cite{AvHeSi, LiLo, ReSi}, which is the Schr\"odinger operator for a particle in the presence of an external magnetic field, playing a fundamental role in Quantum Mechanics in the description of the dynamics of a particle in a non-relativistic setting. 
	In addition, the  magnetic Laplacian  turns out to be the limiting case for $s\rightarrow 1^-$ of the magnetic fractional Laplacian (see \cite{SqVo}), just as it happens for the fractional Laplacian and the classical Laplace operator, in the spirit of Bourgain, Brezis, and Mironescu \cite{BBM}.
	
	As far as it concerns the nonlinearity $g$, here we suppose that there exist $\beta_\infty \in \R$ and $f:\Omega\times\R_+\rightarrow \R$ such that
	$$
	g(x,t):=\ \beta_\infty t + f(x,t^2)t\quad \hbox{  a.e. $x\in \Omega$ and for all $t\in\R$},$$
	hence problem \eqref{P} takes the form
	\begin{equation}
		\label{Pinfty}\tag{$P_{A,\infty}$}
		\begin{cases}
			(-\Delta)^s_Au=\beta_\infty u + f(x, |u|^2)u  & \mbox{in }  \Omega\\
			u=0 & \mbox{in }  \R^N\setminus \Omega.
		\end{cases}
	\end{equation}
	
	Now, problem \eqref{Pinfty} is a perturbation of the eigenvalue problem and, following \cite{FiVe}, in Subection \ref{nonlocal} we recall some features
	about the spectrum of the integro-differential operator $(-\Delta)^s_A$,
	which are very closed to the well known ones concerning the classic Laplace operator and the fractional Laplacian (see e.g. \cite{SeVa}). \\
	Hereafter we denote respectively by $\sigma((-\Delta)^s_A)$ and $(\beta_m^s)_m$
	the spectrum and the non--decreasing, diverging sequence of the eigenvalues of the operator  $(-\Delta)^s_A$, repeated according to their multiplicity (see \cite{FiVe}).
	
	Now we state our main results, referring to \cite[Theorem 4.12]{Ra}, \cite[Theorems 0.1 and 0.3]{bbf} and  \cite[Theorem 3.1]{[BCS]} for the case of the classic Laplace operator.
	\begin{theorem}\label{main}
		Let $s\in  ]0,1[$, $N>2s$, $\Omega$ be an open bounded subset of $\RR^N$ with Lipschitz boundary.
		Assume $\beta_\infty\not\in\sigma((-\Delta)^s_A)$ and that
		\begin{enumerate}[label=$(f_\arabic*)$,ref=f_\arabic*]
			\item \label{f1} $f$ is a Carath\'eodory function and $\displaystyle \sup_{|t|\leq a}|f(\cdot,t^2)t|\in L^\infty(\Omega)$ for all $a>0$;
			\item \label{f2} there exists $\displaystyle \lim_{t\to + \infty}{f(x,t)}\ =\ 0$ uniformly with respect to a.e. $x \in \Omega$.
		\end{enumerate}
		Then, problem \eqref{Pinfty} has at least a weak solution.
		
	\end{theorem}

	\begin{theorem}\label{main1}
		Under the assumptions of Theorem \ref{main}, assume also that
		\begin{enumerate}[label=$(f_\arabic*)$,ref=f_\arabic*]
			\setcounter{enumi}{2}
			\item \label{f3} $\displaystyle \lim_{t\rightarrow 0}{f(x,t)} \ = \beta_0 \in\RR\setminus\{0\}$
			uniformly with respect to a.e. $x \in \Omega$;
		\end{enumerate}
		\begin{enumerate}[label=$(\beta)$,ref=\beta]
			\item \label{beta} there exist $h,k\in\N$, with $k\ge h$, such that $\beta_0 + \beta_\infty<\ \beta^s_h\leq \ \beta^s_k\ <\beta_\infty$.
		\end{enumerate}
		Then, problem \eqref{Pinfty} has at least $k-h+1$ distinct pairs of non--trivial weak solutions.
	\end{theorem}
	We will show that Theorem \ref{main} is a direct consequence of the Saddle Point Theorem (see \cite[Theorem 4.6]{Ra}), while the proof of Theorem \ref{main1} is based on the application of an abstract critical point theorem in \cite[Theorem 2.9]{bbf} that we recall in Section \ref{tools} for the reader convenience.
	\begin{remark}
		\rm{Observe that
			\begin{enumerate}[label=(\roman*),ref=\roman*]
				\item the  statement in  Theorem \ref{main1} holds with slight changes in the proof also if, instead of condition $(\ref{beta})$, we require $\beta_\infty<\ \beta^s_h\leq \ \beta^s_k\ < \beta_0 + \beta_\infty$ (see \cite[Theorem 3.1]{[BCS]});
				\item if  $\beta_0$ in $(\ref{f3})$ belongs to $\{\pm\infty\}$, then we can reason as in \cite[Remark 3.3]{[BCS]};
				\item we remind to \cite[Remarks 1.5 and 3.2]{[BCS]} for some remarks about the case $\beta_0=0$;
				\item in case of resonance, i.e. if $\beta_\infty\in\sigma((-\Delta)^s_A)$, we can proceed as in \cite[Theorem 1.2]{[BCS]}, up to add further assumptions;
				\item the  statement of  Theorem \ref{main} is a  particular case of \cite[Theorem 1]{Fisce} (see also \cite[Remark 3.2]{BaMoBi});
				\item for $A\equiv 0$ we refer to \cite[Theorems 1.2, 1.4]{BaMoBi} (see also \cite{BaMoBi2} and references therein for further related results).
		\end{enumerate}}	
	\end{remark}
	Actually, our results are new also for the {\em local} magnetic Laplacian $(\nabla - iA)^2$. Indeed, denoting by $(\beta_m)_m$ the sequence of its eigenvalues (see Subsection \ref{local} for details), arguing as for the {\em nonlocal} operator, we can prove	
	\begin{theorem}\label{main2}
		Let $\Omega$ be an open bounded subset of $\RR^N$ with Lipschitz boundary.
		Assume $\beta_\infty\not\in\sigma((\nabla - iA)^2)$ and that $(\ref{f1})$ and $(\ref{f2})$ hold. Then, problem
		\begin{equation}
			\label{Plocinfty}\tag{$P^{{\rm loc}}_{A,\infty}$}
			\begin{cases}
				(\nabla - iA)^2u=\beta_\infty u + f(x, |u|^2)u  & \mbox{in }  \Omega\\
				u=0 & \mbox{on }  \partial\Omega.
			\end{cases}
		\end{equation}
		has at least a weak solution.
		
		Moreover, assume that  $(\ref{f3})$ holds and  $(\ref{beta})$ holds  with
		$\beta^s_h$ and $\beta^s_k$ replaced respectively by $\beta_{h}$ and $\beta_{k}$.
		Then, problem \eqref{Plocinfty} has at least $k-h+1$ distinct pairs of non--trivial weak solutions.
	\end{theorem}
		
		This paper is organized as follows: in Section \ref{notations} we recall some properties about the spectrum of the (local) magnetic Laplacian (Subsection \ref{local}),   depict the main aspects of our non--local setting (Subsection \ref{nonlocal}) and present some abstract tools (Subsection \ref{tools}); then,  in Section \ref{principale} we prove Theorems \ref{main} and \ref{main1}.

		\vspace{0,5cm}
		{\bf Notations}
		\begin{itemize}
			\item $B_r(x)$ is the ball in $\RR^N$ of center $x$ and radius $r$;
			\item $\mathcal{R}\, z$, $\bar z$ and $|z|$ are respectively the real part, the complex conjugate and the modulus of a given $z\in\CC$;
			\item $L^2(\Omega,\CC)$ denotes the Lebesgue space of measurable functions $u:\Omega\rightarrow\CC$ such that
			$$
			|u|^2_2=\int_\Omega|u(x)|^2\,dx<+\infty,
			$$
			being $|\cdot |$ the Euclidean norm in $\CC$, endowed with the real scalar product
			$$
			\<u,v\>_2:=\mathcal{R}\int u\bar v\, dx\quad\hbox{for all } u,v\in L^2(\Omega,\CC);$$
			\item the standard norm of $L^p$ spaces is denoted by $|\cdot|_p$;
			\item $H^1_0(\Omega)$ denotes the Sobolev space $W_0^{1,2}(\Omega, \CC)$;
			\item  $(o_m(1))_m$ denotes any infinitesimal sequence.
		\end{itemize}

		\section{Tools and functional framework}\label{notations}
		
		In this section we introduce our functional setting  and some tools needed in the proofs of Theorems \ref{main} and \ref{main1}.

		\subsection{(Local) Magnetic Laplacian}\label{local}
		Given a $L^\infty_{\rm loc}$-vector potential $A$, let us consider the magnetic Laplacian $(\nabla - iA)^2$ in $\Omega$.
		By standard arguments, see e.g. \cite{La}, \cite{LaLiRo}, it can be proved that, considering the zero Dirichlet boundary condition, there exists an othonormal basis $(u_m)_m\subset H^1_0(\Omega)$ of eigenfunctions of the magnetic Laplacian with associated sequence of eigenvalues $(\beta_m)_m$  such that
		$$
		0< \beta_1\leq\beta_2\leq\ldots \beta_m\leq\ldots
		\quad
		\text{and}
		\quad
		\beta_m\rightarrow +\infty\,\,\mbox{as }\,\,m\rightarrow +\infty,$$
		with
		$$
		\beta_m=\frac{\displaystyle\int_\Omega|(\nabla - iA)u_m|^2\,dx}{\displaystyle\int_\Omega|u_m|^2\,dx}\qquad m\in\NN.
		$$
		Let us point out that the eigenvalues - unlike the eigenfunctions - do not change by {\em gauge invariance}, see e.g. \cite[p. 46]{La}, \cite[Appendix A]{LaLiRo} and that, denoted by $\lambda_1$ the first eigenvalue of the Laplace operator with zero Dirichlet boundary condition, it results $\beta_1\geq\lambda_1$ (see \cite[Theorem 10.4]{La}).
		For any $A\in L^\infty_{{\rm loc}}(\R^N\R^N)$,  let us consider the semi--norm
		$$
		[u]^2_{H^1_A(\Omega)}=\int_\Omega|\nabla u-iA(x)u|^2\,dx
		$$
		and, as in \cite{LiLo}, the space
		$$
		H_A^1(\Omega):=\{u\in L^2(\Omega,\CC): [u]_{H^1_A(\Omega)}<+\infty\}
		$$
		endowed with the norm
		$$
		\|u\|^2_{H^1_A(\Omega)}:=|u|_2^2+[u]^2_{H^1_A(\Omega)}.
		$$

		\subsection{(Non--local) Magnetic Laplacian}\label{nonlocal}
		The fractional counterpart of Subsection \ref{local} can be found in \cite{SeVa} for $A\equiv 0$ and in the general case in \cite[Section 3]{FiVe}. Next we highlight the main features.\\
		As for the classical definition of $H^s(\Omega)$, for any $s\in (0,1)$, let us consider the space
		$$
		H^s_A(\Omega):=\{u\in L^2(\Omega,\CC): [u]_{H^s_A(\Omega)}<+\infty\},
		$$
		where
		$$
			[u]_{H^s_A(\Omega)}=\left(\frac{c_{N,s}}{2}\iint_{\Omega\times\Omega}\frac{|u(x)-e^{i(x-y)\cdot A(\frac{x+y}{2})}u(y)|^2}{|x-y|^{N+2s}}\, dxdy\right)^{1/2},
			$$
			endowed with the norm
			\begin{equation}\label{nor}
				\|u\|_{H^s_A(\Omega)}:=\Big(|u|_2^2+[u]^2_{H^s_A(\Omega)}\Big)^{1/2}.
			\end{equation}
			Moreover, denoted by $H^s_A(\RR^N)$ the closure of $C^\infty_0(\R^N)$ with respect to the norm \eqref{nor},	following \cite{FiVe} (see also \cite{SeVa}), let us consider the functional space
			$$
			X_{0,A}:=\{u\in H^s_A(\RR^N): u=0 \hbox{ a.e. in }  \RR^N\setminus\Omega\}
			$$
			and define as in \cite{DaSq} the scalar product
			\begin{equation}\label{scapro}
				\<u,v\>_{X_{0,A}}:=\frac{c_{N,s}}{2}\mathcal{R}\iint_{\RR^{N}\times\RR^{N}}\frac{\left(u(x)-e^{i(x-y)\cdot A(\frac{x+y}{2})}u(y)\right)\overline{\left(v(x)-e^{i(x-y)\cdot A(\frac{x+y}{2})}v(y)\right)}}{|x-y|^{N+2s}}\, dxdy.
			\end{equation}
			The norm $\|u\|_{X_{0,A}}:=\sqrt{\<u,u\>_{X_{0,A}}}$ is equivalent to \eqref{nor}  in  $H^s_A(\R^N)$ (see \cite[Lemma 2.1]{FiPiVe})
			and $(X_{0,A},\<\cdot,\cdot\>_{X_{0,A}})$ is a real separable Hilbert space (see \cite[Lemma 7]{ReSi}).\\
			Being $\Omega$ open and bounded, $X_{0,A}\hookrightarrow H^s(\Omega)$, and, since $\partial\Omega$ is Lipschitz, $X_{0,A}\hookrightarrow\hookrightarrow L^p(\Omega,\CC)$ for any $p\in [1,2_s^\ast)$, with $2^\ast_s:=\frac{2N}{N-2s}$ (see \cite[Lemma 2.2]{FiPiVe}).\\
			A function $u\in X_{0,A}$ is a {\em weak solution} of \eqref{Pinfty} if and only if, for all $\varphi\in X_{0,A}$,
			$$
			\<u,\varphi\>_{X_0,A} = \beta_\infty\mathcal{R}\int_\Omega  u\,\bar\varphi \, dx \,
			\displaystyle +
			\mathcal{R}\int_\Omega f(x,|u|^2)u\bar\varphi\;dx.
			$$
			Following \cite{FiVe}, we call {\em variational Dirichlet eigenvalue}, or simply {\em eigenvalue}, a value $\beta\in\RR$ for which there exists a nontrivial weak solution $u\in X_{0,A}$, called {\em eigenfunction}, of
			\begin{equation}
				\label{autovalore debole}
				\begin{cases}
					(-\Delta)^s_Au=  \beta u & \mbox{ in }  \Omega\\
					u=0 & \mbox{ in }  \R^N\setminus \Omega.
				\end{cases}
			\end{equation}
			In \cite{FiVe} it is proved that eigenfunctions of \eqref{autovalore debole} corresponding to different eigenvalues are orthogonal with respect to \eqref{scapro} (see \cite[Lemma 3.2]{FiVe}) and that, proceeding by induction, it is possible to show that there exists a sequence $(\beta^s_m)_m\subset\RR$ of eigenvalues of \eqref{autovalore debole} and a sequence $(f_m)_m\subset X_{0,A}$ of associated eigenfunctions such that
			\begin{equation}\label{lambdak+1}
				\beta^s_1=\min_{u\in X_{0,A}\setminus\{0\}}\frac{\|u\|^2_{X_{0,A}}}{|u|_2^2}
				\quad
				\text{and}
				\quad
				\beta^s_{m+1}=\min_{u\in  \mathbb E_{m+1}\setminus\{0\}}\frac{\|u\|^2_{X_{0,A}}}{|u|_2^2}
				\quad \hbox{ for any } m\in\N^*,
			\end{equation}
			where
			$\mathbb{E}_{1}=X_{0,A}$,
			$$
			\mathbb{E}_{m+1}:=\{u\in X_{0,A}:\<u,f_j\>_{X_{0,A}}=0\hbox{ for every } j =1,\ldots, m\}
			$$
			and $f_1\in X_{0,A}$, $f_{m+1}\in\mathbb E_{m+1}$ for $m\geq 1$ attain the minima in \eqref{lambdak+1} (see \cite[Proposition 3.3]{FiVe}).
			
			Moreover, the eigenfunctions $f_m$ are orthogonal also with respect to the real $L^2$-scalar product (see \cite[Proposition 3.4]{FiVe}) and the eigenvalues $\beta^s_m$ satisfy
			\begin{equation}\label{ordine lambdak}
				0<\beta^s_1\leq\beta^s_2\leq\ldots\leq\beta^s_m\leq\ldots
				\quad
				\text{and}
				\quad
				\beta^s_m\rightarrow +\infty\,\,\mbox{as }\,\,m\rightarrow +\infty
			\end{equation}
			(see \cite[Proposition 3.5]{FiVe}). Furthermore $(f_m)_m$ is an orthonormal basis of $L^2(\Omega,\CC)$ and an orthogonal one of $X_{0,A}$ (\cite[Proposition 3.7]{FiVe}) and  $\beta^s_m$ is an eigenvalue with finite multiplicity for each $m\in\N$ (see \cite[Proposition 3.8]{FiVe}).
			
			Denoting for any $m\in\N^*$ by $\mathbb H_m:={\rm span}\{f_1,\ldots,f_m\}$, it results (with respect to \eqref{scapro})
			\begin{equation*}
				X_{0,A}=   \mathbb H_m\oplus \mathbb{E}_{m+1} \quad\hbox{ and }  \quad\mathbb{E}_{m+1}=\mathbb{H}_m^\perp=\overline{\displaystyle{\rm span}\{f_j: j\geq m+1\}}.
			\end{equation*}
			Moreover, for any $m \in\N^*$ the $m$--eigenvalue can be characterized as
			$$
			\beta^s_{m}=\max_{u\in \mathbb H_{m}\setminus{\left\{0\right\}}}\frac{\|u\|^2_{X_{0,A}}}{|u|_2^2}
			$$
   following \cite[Chapter 8]{MoRaSe}.

			\subsection{An abstract critical point theorem}\label{tools}
			Now, let $(X,\|\cdot\|_X)$ be a Banach space,
			$(X',\|\cdot\|_{X'})$ its dual, $I$ a $C^1$-functional on $X$, $I'$ its differential.
			The func\-tio\-nal $I$ satisfies the {\em Palais--Smale condition at level $c$} ($c \in \R$) if any sequence $(u_m)_m \subseteq X$ such that
			\begin{equation}\label{cer}
				\lim_{m \to +\infty}I(u_m) = c\quad\mbox{and}\quad
				\lim_{m \to +\infty}\|I'(u_m)\|_{X'} = 0
			\end{equation}
			converges in $X$, up to subsequences. If $-\infty\leq a<b\leq +\infty$, $I$ satisfies the Palais--Smale condition in $]a,b[$ if so is at each level $c\in ]a,b[$.\\

			%
			We will use the following abstract critical point theorem whose proof  is based on the pseudo--index related to the genus
			(see \cite{b} for more details). \\
			\begin{theorem}[\cite{bbf}, Theorem 2.9]\label{group}
				
				Let $I\in C^1(X,\R)$ and assume that:
				\begin{enumerate}
					\item $I$ is even;
					\item $I$ satisfies the Palais--Smale condition in $\R$;
					\item there exist two closed subspaces $V,W\subset X$ such that  $\dim V<+\infty, {\rm codim}\, W<+\infty$ and two constants $c_0,c_\infty$, such that $c_\infty>c_0$, verifying the following assumptions:
					\begin{itemize}
						\item $I(u)\geq c_0 $ on $S_\rho\cap W$ (resp. on $S_\rho\cap V$), 	where $S_\rho=\{u\in X: \|u\|_X=\rho\}$;
						\item $I(u)\leq c_\infty $ on $V$  (resp. on $W$).
					\end{itemize}
				\end{enumerate}
				
				If, moreover, $\dim V >  {\rm codim}\, W$, then $I$ has at least $\dim V -  {\rm codim}\, W$ distinct pairs of critical points whose corresponding critical values belong to $[c_0,c_\infty]$.

			\end{theorem}
			
			%

			\section{Proof of Theorem \ref{main}}\label{principale}

			By
			$(\ref{f1})$  and $(\ref{f2})$, for all $\varepsilon>0$ there exists $a_\varepsilon>0$ such that
			\begin{equation}\label{infymin}
				|f(x,t^2)t|\leq\varepsilon |t| + a_\varepsilon,
				\quad \hbox{ for a.e. } x\in \Omega, \hbox{ for all }  t\in\R.
			\end{equation}
			
			The weak solutions of problem \eqref{Pinfty}  are the critical points of the $C^1$--functional
			\begin{equation*}*\label{1bis}
				J_A(u)
				:=
				\|u\|^2_{X_{0,A}}
				-\beta_\infty |u|_2^2
				- \int_\Omega F(x,|u|^2) \;dx,
			\end{equation*}
			defined in $X_{0,A}$, with $F(x,t):=\displaystyle\int_0^t f(x,s) \; ds$ and, for every $u,\varphi\in X_{0,A}$,
			$$
			J'_A(u)[\varphi]=
			\<u,\varphi\>_{X_0} - \beta_\infty\mathcal{R}\int_\Omega  u\,\bar\varphi \;dx \, -
			\mathcal{R}\int_\Omega f(x,|u|^2)u\, \bar\varphi\;dx.
			$$

			Under the assumptions of Theorems \ref{main} and \ref{main1}, the functional $J_A$ satisfies the following compactness property.
			\begin{lemma}\label{palmsmin}
				Assume that $(\ref{f1})$ and $(\ref{f2})$ hold. Then,
				if $\beta_\infty\not\in\sigma((-\Delta)^s_A)$, the functional $J_A$ satisfies the Palais--Smale condition in $\RR$.
			\end{lemma}
			\begin{proof}
				Let  $c\in\R$ and $(u_m)_m$ be a sequence in $X_{0,A}$ such that \eqref{cer} holds.\\
				To prove that $(u_m)_m$ is bounded in $X_{0,A}$, arguing by contradiction, we assume that $\|u_m\|_{X_{0,A}}\rightarrow +\infty$ as $m\rightarrow + \infty$. Since $(w_m)_m:=( u_m/{\|u_m\|_{X_{0,A}}})_m$  is bounded in $X_{0,A}$, there exists $w\in X_{0,A}$ such that, up to a subsequence, $(w_m)_m$ converges to $w$
				weakly in $X_{0,A}$ and strongly in $L^{2}(\Omega,\CC)$.\\
				Using the boundedness of $(w_m)_m$, by \eqref{cer} we get
				\begin{equation}\label{ro8min}
					\begin{split}
						o_m(1)
						&=
						J_A'(u_m)[(w_m-w)/\|u_m\|_{X_{0,A}}]\\
						&= \<w_m,w_m-w\>_{X_{0,A}}
						- \beta_\infty\mathcal{R}\int_\Omega w_m\,\overline{(w_m-w)} \;dx
						- \mathcal{R}
						\int_\Omega \frac{f(x,|u_m|^2)u_m}{\|u_m\|_{X_{0,A}}}\overline{(w_m-w)}\;dx.
					\end{split}
				\end{equation}
				Moreover, by the convergence of $(w_m)_m$ to $w$ in $L^{2}(\Omega,\CC)$, it follows that
				$$
				\displaystyle{\left| \int_\Omega w_m \overline{(w_m-w)} \;dx\right|\leq
					\int_\Omega |w_m||w_m-w| \;dx\leq
					|w_m|_2|w_m-w|_2=o_m(1)}
				$$
				and, using also  \eqref{infymin}, we infer that
				$$
				\left| \int_\Omega \frac{f(x,|u_m|^2)u_m}{\|u_m\|_{X_{0,A}}}\overline{(w_m-w)}\;dx \right|\leq \varepsilon
				|w_m|_2|w_m-w|_2 + \frac{a_\varepsilon |w_m-w|_1}{\|u_m\|_{X_{0,A}}} =o_m(1).
				$$
				Then, by \eqref{ro8min}, it follows that $\<w_m, w_m-w\>_{X_{0,A}} =o_m(1)$, therefore $w_m\rightarrow  w$ in $X_{0,A}$ and $w\not=0$.\\
				Now, reasoning as in \eqref{ro8min}, for all $\varphi\in X_{0,A}$
				\begin{equation}\label{ro8xxmin}
					\begin{split}
						o_m(1)
						&=
						J_A'(u_m)[\varphi/\|u_m\|_{X_{0,A}}]\\
						&= \<w_m,\varphi\>_{X_{0,A}}
						- \beta_\infty\mathcal{R}\int_\Omega w_m\,\overline{\varphi} \;dx
						- \mathcal{R}
						\int_\Omega \frac{f(x,|u_m|^2)u_m}{\|u_m\|_{X_{0,A}}}\overline{\varphi}\;dx.
					\end{split}
				\end{equation}
				Moreover, for every $\varepsilon>0$ let $\tilde{m}_\varepsilon\in\N$ be such that, for every $m\geq \tilde{m}_\varepsilon$, $\|u_m\|_{X_{0,A}}>a_\varepsilon/\varepsilon$, where $a_\varepsilon>0$ is given in \eqref{infymin}. Then, using \eqref{infymin}, for every $\varepsilon>0$ and $m\geq \tilde{m}_\varepsilon$,
				$$
				\left| \int_\Omega \frac{f(x,|u_m|^2)u_m}{\|u_m\|_{X_{0,A}}}\overline{\varphi}\;dx \right|\leq \varepsilon
				|w_m|_2|\varphi|_2 + \frac{a_\varepsilon |\varphi|_1}{\|u_m\|_{X_{0,A}}}
				\leq C\varepsilon,
				$$
				for a suitable $C>0$, so that, for every $\varphi\in X_{0,A}$,
				\begin{equation*}
					\lim_{m\rightarrow + \infty}\mathcal{R}\int_\Omega \frac{f(x,|u_m|^2)u_m}{\|u_m\|_{X_{0,A}}}\overline\varphi\;dx =0.
				\end{equation*}
				Thus, passing to the limit in \eqref{ro8xxmin}, we get that, for every $\varphi\in X_{0,A}$,
				$$
				\<w,\varphi\>_{X_{0,A}}  = \beta_\infty\mathcal{R}\int_\Omega w\overline\varphi \;dx,
				$$
				namely $\beta_\infty\in\sigma((-\Delta)^s_A)$, against our assumption.\\
				Hence $(u_m)_m$ is bounded in $X_{0,A}$ and, due to the reflexivity of our space, there exists $u_0\in X_{0,A}$ such that, up to a subsequence, $u_m\rightharpoonup u_0$ in $X_{0,A}$, $u_m\rightarrow  u_0$ in $L^2(\Omega,\CC)$,
				$$
				\langle J'_A(u_m), u_m- u_0\rangle \rightarrow  0 \quad \hbox{ as } m\rightarrow +\infty
				$$
				and, by \eqref{infymin},
				$$
				\left|\int_\Omega f(x,|u_m|^2)u_m\overline{(u_m-u_0)}\;dx\right|\rightarrow 0  \quad \hbox{ as } m\rightarrow +\infty.
				$$
				Therefore, reasoning as before, $u_m\rightarrow u_0$ in $X_{0,A}$ and the proof is complete.
			\end{proof}

			Now we are ready to prove our main results.
			
			\begin{proof}[Proof of Theorem \ref{main}]
				If $\beta^s_1 < \beta_\infty$, the statement follows by a standard application of the Saddle Point Theorem (see \cite[Theorem 4.6]{Ra}). In particular, due to Lemma \ref{palmsmin}, we only need to check its {\em geometrical} assumptions.\\
				First of all observe that, for every $m\in\mathbb{N}^*$,
				\begin{equation}
					\label{findim}
					\|u\|_{X_{0,A}}^2\leq \beta_m^s |u|_2^2
					\text{ for every } u\in \mathbb{H}_m
				\end{equation}
				and
				\begin{equation}
					\label{infdim}
					\|u\|_{X_{0,A}}^2\geq \beta_m^s |u|_2^2
					\text{ for every } u\in \mathbb{E}_m.
				\end{equation}
				Then, using \eqref{ordine lambdak}, let us consider $\nu\in\mathbb{N}^*$ such that $\beta^s_{\nu} < \beta_\infty < \beta^s_{\nu +1}$ and
				\begin{equation}
					\label{condeps}
					\varepsilon\in \big(0,\min\{\beta_\infty- \beta^s_{\nu}, \beta^s_{\nu+1}-\beta_\infty \}\big).
				\end{equation}
				Thus, by  \eqref{infymin} and \eqref{infdim}, for every $\varepsilon$ there exists $C_\varepsilon>0$ such that for every $u\in\mathbb{E}_{\nu+1}$
				\[
				J_A(u)
				\geq
				\|u\|_{X_{0,A}}^2
				- \beta_\infty |u|_2^2
				-\varepsilon |u|_2^2
				-C_\varepsilon \|u\|_{X_{0,A}}
				\geq
				\left(1 -\frac{\beta_\infty+\varepsilon}{\beta^s_{\nu+1}}\right)\|u\|_{X_{0,A}}^2
				-C_\varepsilon \|u\|_{X_{0,A}}
				\]
				so that, using \eqref{condeps}, we obtain that there exists $\alpha_1 >0$ such that, for all $u\in\mathbb{E}_{\nu+1}$, $J_A(u)\geq -\alpha_1$, namely \cite[($I_4$) of Theorem 4.6]{Ra}.\\
				Moreover, by \eqref{infymin}, \eqref{findim} and \eqref{condeps}, for all  $u\in \mathbb{H}_{\nu}$
			\begin{equation*}
				J_A(u)
				\leq
				\|u\|_{X_{0,A}}^2
				-\beta_\infty |u|_2^2
				+ \varepsilon |u|^2_2
				+ C_\varepsilon \|u\|_{X_{0,A}}
				\leq
				\left(1 -\frac{\beta_\infty-\varepsilon}{\beta^s_{\nu}}\right) \|u\|_{X_{0,A}}^2
				+ C_\varepsilon\|u\|_{X_{0,A}}
				\to -\infty
			\end{equation*}
			as $\|u\|_{X_{0,A}}\to +\infty$,  and so, for $\rho>0$ large enough, that $J_A(u)\leq -\alpha_2$ for all $u\in \mathbb{H}_{\nu} \cap
			S_\rho$ where $S_\rho:=\{u\in X_{0,A}: \|u\|_{X_{0,A}}= \rho\}$, with $\alpha_2>\alpha_1$, getting \cite[($I_3$) of Theorem 4.6]{Ra}.\\
			On the other hand, by \eqref{lambdak+1}, if $0<\beta_\infty< \beta^s_1$,
			\[
			J_A(u)
			\geq
			\|u\|_{X_{0,A}}^2
			- \beta_\infty |u|_2^2
			-\varepsilon |u|_2^2
			-C_\varepsilon \|u\|_{X_{0,A}}
			\geq
			\left(1 -\frac{\beta_\infty+\varepsilon}{\beta^s_{1}}\right)\|u\|_{X_{0,A}}^2
			-C_\varepsilon \|u\|_{X_{0,A}}
			\]
			and, if $\beta_\infty\leq 0$,
			\[
			J_A(u)
			\geq
			\|u\|_{X_{0,A}}^2
			- \beta_\infty |u|_2^2
			-\varepsilon |u|_2^2
			-C_\varepsilon \|u\|_{X_{0,A}}
			\geq
			\left(1 -\frac{\varepsilon}{\beta^s_{1}}\right)\|u\|_{X_{0,A}}^2
			-C_\varepsilon \|u\|_{X_{0,A}}
			\]
			for all $u\in X_{0,A}$. Thus, in both cases, if $\varepsilon\in(0,\beta^s_1-\beta_\infty)$ and $\varepsilon\in(0,\beta^s_1)$ respectively, we have that the functional $J_A$ is bounded from below and so we get a weak solution for problem \eqref{Pinfty} by a direct minimization argument.

			\end{proof}

			\begin{proof}[Proof of Theorem \ref{main1}]

				To get the statement, first observe that, by $(\ref{f2})$ and $(\ref{f3})$,
				$$
				\lim_{|t|\rightarrow +\infty}\frac{F(x,t^2)}{t^{2}} \ = 0
				\quad
				\text{and}
				\quad
				\lim_{t\rightarrow 0}\frac{F(x,t^2)}{t^{2}} \ = {\beta_0},
				$$
				uniformly with respect to a.e. $x \in \Omega$. Therefore, for every $\varepsilon>0$ there exist $r_\varepsilon>\delta_\varepsilon>0$ such that
				\begin{equation*}
					|F(x,t^2)|\leq \varepsilon t^2, \text{ if } |t|> r_\varepsilon
					\quad\text{and}\quad
					\left|F(x,t^2)-{\beta_0}t^2\right|\leq \varepsilon t^2,   \text{ if } |t|< \delta_\varepsilon,
				\end{equation*}
				for a.e. $x\in \Omega$.\\
				Moreover, by $(\ref{f1})$, taking any $q\in ]0, 4s/(N-2s)]$,  there exists $c_{r_\varepsilon}>0$ such that
				\begin{equation*}
					|F(x,t^2)|
					\leq c_{r_\varepsilon}{|t|}^{q+2}, \quad \text{ if } \delta_\varepsilon\leq |t|\leq r_\varepsilon \text{ and for a.e. } x\in\Omega.
				\end{equation*}
				Thus, for any $\varepsilon>0$, there exists $c_\varepsilon>0$ such that
				$$
				F(x,t^2)
				\leq (\beta_0 + \varepsilon) t^2 + {c_\varepsilon}|t|^{q+2},
				\quad
				\text{for a.e. } x\in\Omega \text{ and for all } t\in\R,
				$$
				so that, for all $u\in X_{0,A}$, using
				Sobolev inequalities,
				\begin{equation*}
					J_A(u)
					\geq
					\|u\|_{X_{0,A}}^2
					- (\beta_\infty + \beta_0 + \varepsilon)|u|_2^2
					- c_\varepsilon' \|u\|_{X_{0,A}}^{q+2} \quad\hbox{ for all } u\in X_{0,A},
				\end{equation*}
				for a suitable $c'_\varepsilon>0$.\\	
				Therefore, given $h$ as in $(\ref{beta})$, by \eqref{infdim} we obtain that, for $\varepsilon\in(0,\beta^s_h-(\beta_\infty+\beta_0))$, there exists $\tilde{c}_\varepsilon>0$ such that, for every $u\in \mathbb{E}_h$,
				\[
				J_A(u)
				\geq
				\left(1 - \frac{\beta_\infty + \beta_0 + \varepsilon}{\beta^s_{h}}\right)\|u\|_{X_{0,A}}^2
				- c_\varepsilon' \|u\|_{X_{0,A}}^{q+2}
				\geq
				\tilde{c}_\varepsilon\|u\|_{X_{0,A}}^2
				- c_\varepsilon'\|u\|_{X_{0,A}}^{q+2}.
				\]		
				Hence we can conclude that, if $\rho$ is small enough, there exists $c_0>0$ such that
				\begin{equation}\label{merco}
					J_A(u)\geq c_0\quad\hbox{ for all } u\in S_\rho\cap \mathbb{E}_{h},
				\end{equation}
				with $S_\rho$ as in the proof of Theorem \ref{main}.\\
				Moreover, taking $k$ as in $(\ref{beta})$, $\varepsilon\in (0,(\beta_\infty-\beta^s_k)/\beta^s_k)$, and using \eqref{infymin} 
				and \eqref{findim}, we have that there exist $C_\varepsilon, c_\infty>0$ such that for all $u\in \mathbb{H}_k$,
					\begin{equation*}
						J_A(u)
						\leq
						\left(1+\varepsilon - \frac{\beta_\infty}{\beta_k^s}\right)\|u\|_{X_{0,A}}^2
						+ C_\varepsilon\|u\|_{X_{0,A}}
						\leq
						c_\infty.			\end{equation*}
					Finally, to have $c_0<c_\infty$ it is enough to take $\varepsilon\in(0,\min\{\beta^s_h-(\beta_\infty+\beta_0),(\beta_\infty-\beta^s_k)/\beta^s_k)\})$ and $\rho$ sufficiently small in \eqref{merco}.\\
					Now we can conclude the proof. Indeed, $J_A$ is even and  by Lemma \ref{palmsmin} it satisfies the Palais--Smale condition in $\R$; moreover	taking $W=\mathbb{E}_{h}$ and $V=\mathbb{H}_{k}$, Theorem 	\ref{group} applies and $J_A$ has at least $k-h+1$ distinct pairs of critical points.

				\end{proof}

				\noindent {\bf Acknowledgements.} {The paper is realized with the auspices of the INdAM - GNAMPA 2023 Projects: {\it Equazioni nonlineari e problemi tipo Calabi--Bernstein} and {\it Metodi variazionali per alcune equazioni di tipo Choquard}, and PRIN projects 2017JPCAPN   {\it Qualitative and quantitative aspects of nonlinear PDEs}, P2022YFAJH {\it Linear and Nonlinear PDEs: New directions and Applications}, and 2022BCFHN2 {\it Advanced theoretical aspects in PDEs and their applications}.}


\begin{thebibliography}{99}
					
					
					
					
					
					\bibitem{Appl}
					D. Applebaum,  L\'evy processes - from probability theory to finance and quantum groups, {\em Notices of the American Mathematical Society} {\bf 51} (2004), 1336-1347.		
					
					\bibitem{AvHeSi} J. Avron, I. Herbst, B. Simon, Schr\"odinger operators with magnetic fields. I. General interactions, {\em Duke Math. J.} {\bf 45} (1978), 847-883.
					
					\bibitem{bbf}
					{P. Bartolo, V. Benci, D. Fortunato,}
					Abstract critical point theorems and applications to some nonlinear problems
					with ``strong'' resonance at infinity, {\em Nonlinear Anal.} {\bf 7} (1983), 981-1012.
					
				
					
					
					\bibitem{[BCS]} {R. Bartolo, A.M. Candela, A. Salvatore}, {Perturbed asymptotically linear problems}, {\em Ann. Mat. Pura Appl.} {\bf 193} (2014), 89-101.
					
					\bibitem{BaMoBi} { R. Bartolo, G. Molica Bisci}, A pseudo--index approach to fractional equations,  {\em  Expo. Math.} {\bf 33} (2015), 502-516.
					
					\bibitem{BaMoBi2} { R. Bartolo, G. Molica Bisci}, Asymptotically linear fractional $p$-Laplacian equations
					 {\em 	Ann. Mat. Pura Appl.}  {\bf 196} (2017), 427-442.
					 
					
					\bibitem{b}
					{V. Benci,}
					On the critical point theory for indefinite functionals in the
					presence of symmetries, {\em Trans. Am. Math. Soc.} {\bf 274} (1982), 533-572.
					
					\bibitem{BBM}
					J. Bourgain, H. Brezis, P. Mironescu, Limiting embedding theorems for $W^{s,p}$ when $s \uparrow 1$ and applications, {\em J. Anal. Math.} {\bf 87} (2002), 77--101.
					
					\bibitem{DaSq} P. d'Avenia, M. Squassina, Ground states for fractional magnetic operators, {\em ESAIM Control Optim. Calc. Var.} {\bf 24} (2018), 1-24.
					
					\bibitem{dpv} { E. Di Nezza, G. Palatucci, E. Valdinoci},
					{Hitchhiker's guide to the fractional Sobolev spaces},
					{\em Bull. Sci. Math.} \textbf{136} (2012), 521-573.
					
					\bibitem{Fisce}{ A. Fiscella}, {Saddle point solutions for non--local elliptic operators},
					 {\em Topol. Methods Nonlinear Anal.} \textbf{44} (2014), 527-538.
					
					
					\bibitem{FiVe} A. Fiscella, E. Vecchi, Bifurcation and multiplicity results for critical magnetic fractional problems, {\em Electron. J. Differential Equations} \textbf{153} (2018), 18 pp.
					
					\bibitem{FiPiVe} A. Fiscella, A. Pinamonti, E. Vecchi, Multiplicity results for magnetic fractional problems, {\em J. Differential Equations} \textbf{263} (2017), 4617-4633.
					
					
					
					
					\bibitem{Ic} T. Ichinose, Magnetic relativistic Schr\"odinger operators and imaginary-time path integrals.
					In: Ball, J. A., Bottcher, A., Dym, H., Langer, H., Tretter, C. (eds.) Operator Theory:
					Advances and Applications, vol. 232, pp. 247-297. Birkh\"auser/Springer, Basel, 2013.
					
					\bibitem{IT} T. Ichinose and H. Tamura, Imaginary-time path integral for a relativistic spinless particle in an electromagnetic field, Commun. Math. Phys. 105 (1986) 239--257.
					
					\bibitem{Las}
					N. Laskin, Fractional Schr\"odinger equation, Phys. Rev. E 66 (2002), 056108.
					
					\bibitem{La} R.S. Laugesen, Spectral Theory of Partial Differential Equations, {\em Spectral theory and applications}, 23-55, Contemp. Math., 720, Centre Rech. Math. Proc., Amer. Math. Soc., Providence, RI, 2018. 
					
					\bibitem{LaLiRo} R.S. Laugesen, J. Liang, A. Roy,  Sums of magnetic eigenvalues are maximal on rotationally symmetric domains, {\em  Ann. Henri Poincar\'e} \textbf{13} (2012), 731--750.
					
					\bibitem{LiLo} E.H. Lieb, M. Loss, {\em Analysis}. Second edition. Graduate Studies in Mathematics, 14. American Mathematical Society, Providence, RI, 2001. xxii+346
					
					\bibitem{MoRaSe} G. Molica Bisci, V.D. Radulescu, R. Servadei,  Variational methods for nonlocal fractional problems. With a foreword by Jean Mawhin, Encyclopedia of Mathematics and its Applications, 162. Cambridge University Press, Cambridge, 2016. xvi+383 pp.
					
					
					\bibitem{Poz}
					C. Pozrikidis, The fractional Laplacian, CRC Press, Boca Raton, FL, (2016).
					
					\bibitem{Ra}
					P.H. Rabinowitz,
					{\sl Minimax Methods in Critical Point Theory with Applications to Differential Equations}, CBMS Regional Conf. Ser. in Math. {\bf 65}, Amer. Math. Soc., Providence, 1984.
					
					\bibitem{ReSi} M. Reed, B. Simon, {\em Methods of modern mathematical physics, I, Functional analysis}, Academic Press, Inc.,
					New York, 1980.
					
					
					
					
					\bibitem{SeVa} R. Servadei, E. Valdinoci, Variational methods for non--local operators of elliptic type, {\em Discrete Contin. Dyn. Syst.}
					\textbf{33} (2013), 2105-2137.
					
					
					\bibitem{SqVo} M. Squassina, B. Volzone,  Bourgain-Br\'ezis-Mironescu formula for magnetic operators, {\em C. R. Math. Acad. Sci. Paris}  {\bf 354} (2016), 825-831.
					
					
					
				\end{thebibliography}
			\end{document}